\documentclass{article}

\usepackage{parskip}

\usepackage[top=70pt,bottom=70pt,left=70pt,right=70pt]{geometry}

\usepackage{amsmath,amssymb,amsthm}

\usepackage{graphicx}

\newcommand{\pa}[1]{\left(#1\right)}
\newcommand{\br}[1]{\left[#1\right]}
\newcommand{\cb}[1]{\left\{#1\right\}}
\newcommand{\abs}[1]{\left\lvert#1\right\rvert}
\newcommand{\nm}[1]{\left\|#1\right\|}
\newcommand{\on}[1]{\operatorname{#1}}

\theoremstyle{plain}
\newtheorem{prp}{Proposition}
\theoremstyle{remark}
\newtheorem{remark}{Remark}

\usepackage[numbers]{natbib}
\makeatletter
\def\NAT@spacechar{~}
\makeatother

\usepackage{enumerate}

\usepackage[hidelinks]{hyperref}

\usepackage{authblk}

\title{The central limit theorem for a sequence of random processes with space varying long memory}
\author{Vaidotas Characiejus\thanks{Corresponding author.} }
\author{Alfredas Ra\v ckauskas}
\affil{Faculty of Mathematics and Informatics, Vilnius University, Naugarduko g. 24, 03225 Vilnius, Lithuania\\ (e-mail addresses: \href{mailto:vaidotas.characiejus@gmail.com}{vaidotas.characiejus@gmail.com}; \href{mailto:alfredas.rackauskas@mif.vu.lt}{alfredas.rackauskas@mif.vu.lt})}
\date{January 18, 2013}

\begin{document}

\maketitle

\begin{abstract}
\noindent In this paper we investigate a sequence of square integrable random processes with space varying memory. We establish sufficient conditions for the central limit theorem in the space $L^2(\mu)$ for the partial sums of the sequence of random processes with space varying long memory.  Of particular interest is a non-standard normalization of the partial sums in the central limit theorem.


\end{abstract}

\section{Introduction}
Memory of  a second-order stationary sequence of random variables  $\cb{Y_k}$ is usually defined in terms of  the decay of  autocovariances $\on{Cov}\br{Y_0, Y_h}$. For example,  $\cb{Y_k}$ is said to have long memory if the series
\begin{equation*}
\sum_{h=0}^\infty\abs{\on{Cov}\br{Y_0,Y_h}}
\end{equation*}
diverges, while short memory of $\cb{Y_k}$ corresponds to the convergence of the series above. For a review of the notion of long memory, we refer to \citet{samorodnitsky2007}; for probabilistic foundations, statistical methods, and applications, we refer to \citet*{giraitis2012}, and \citet{beran1994}.

Interesting and important  features of memory are reflected in the growth rate of the partial sums $\sum_{k=1}^nY_k$. The partial sums $\sum_{k=1}^nY_k$ of a short memory sequence of random variables appear to grow at the rate of the central limit theorem $n^{1/2}$, whereas the partial sums of a long memory sequence of random variables may grow faster. For example, if $\on{Cov}\br{Y_0, Y_h}\sim h^{-d}, 0<d<1$, then the partial sums $\sum_{k=1}^nY_k$  grow at the rate of $n^{1-d/2}$.

Memory of a sequence of multidimensional random elements may vary in space. That is, it can depend on the direction the sequence is projected. Memory of  a second-order stationary sequence of Hilbert space  valued random elements  $\cb{Y_k}$ can be associated with the regular decay of the autocovariance operators  $Q_h\mathrel{\mathop:}=\on{Cov}\br{Y_0, Y_h}$ by assuming, for instance, $Q_h \sim h^{-D}Q$, where $D$ and $Q$ are certain operators on the Hilbert space (see \citet{rackauskas2011} for further details). The corresponding sequence has space varying memory which depends on the operator $D$.

We investigate a sequence of random processes $\cb{X_k}\mathrel{\mathop:}=\cb{X_k(t),t\in S}$ which is defined for each $k\in\mathbb Z$ and each $t\in S$ by
\begin{equation*}
X_k(t)\mathrel{\mathop:}=\sum_{j=0}^\infty\pa{j+1}^{-d(t)}\varepsilon_{k-j}(t),
\end{equation*}
where $S$ is some index set, $\cb{\varepsilon_k(t)}$ is a sequence of independent and identically distributed random variables and $d(t)$ is a real function of $t$.

The sequence $\cb{X_k(t)}$ for each $t\in S$ is essentially similar to the fractional ARIMA$\pa{0,1-d(t),0}$ process, which may be expressed as an MA$\pa{\infty}$ process with the coefficients
\begin{equation*}
\frac{\Gamma\pa{j+1-d(t)}}{\Gamma\pa{j+1}\Gamma\pa{1-d(t)}},\quad j=0,1,2\ldots,
\end{equation*}
where $\Gamma(\cdot)$ is the gamma function (the fractional ARIMA process was introduced by \citet{granger1980} and \citet{hosking1981}). The application of Stirling's formula to the coefficients above yields the following relation:
\begin{equation*}
\frac{\Gamma\pa{j+1-d(t)}}{\Gamma\pa{j+1}\Gamma\pa{1-d(t)}}\sim\frac{j^{-d(t)}}{\Gamma\pa{1-d(t)}}\quad\text{as}\quad j\to\infty.
\end{equation*}

The growth rate of the partial sums $\sum_{k=1}^nX_k(t)$ depends on $t$. Interpreting $k\in\mathbb Z$ as a time index and $t\in S$ as a space index we thus have a sequence of real valued random processes $\cb{X_k}=\cb{X_k(t), t\in S}$ with space varying memory. Such sequences of random processes could serve as a model in functional data analysis (we refer to \citet{ramsay2005} for an introduction to functional data analysis, for the  theory of linear processes in function spaces, see \citet{bosq2000}).

We investigate the growth of the partial sums $\sum_{k=1}^n X_k$ in the sample path space of square integrable real valued functions and establish sufficient conditions for the central limit theorem for \mbox{the partial sums $\sum_{k=1}^nX_k$.}

Figure \ref{fig:sample paths} shows simulated sample paths of the random processes of the sequence $\cb{X_k}$. The sequence $\cb{\varepsilon_k}\mathrel{\mathop:}=\cb{\varepsilon_k(t):t\in[0,1]}$ was assumed to be a sequence of independent and identically distributed standard Wiener processes on the interval $[0,1]$ and the function $d:[0,1]\to(1/2,+\infty)$ was assumed to be a step function $d(t)=d_1\chi_{[0,1/2)}(t)+d_2\chi_{[1/2,1]}(t)$, where $\chi_A$ is the indicator function of $A$. The simulated sample paths for $5$ consecutive elements of the sequence $\cb{X_k}$ were plotted. The procedure was completed for two different sets of the values of $d_1$ and $d_2$ ($d_1=0.6, d_2=2$ and $d_1=0.6,d_2=0.7$).

\begin{figure}[h]
\label{fig:sample paths}
\centering
\includegraphics[width=\textwidth]{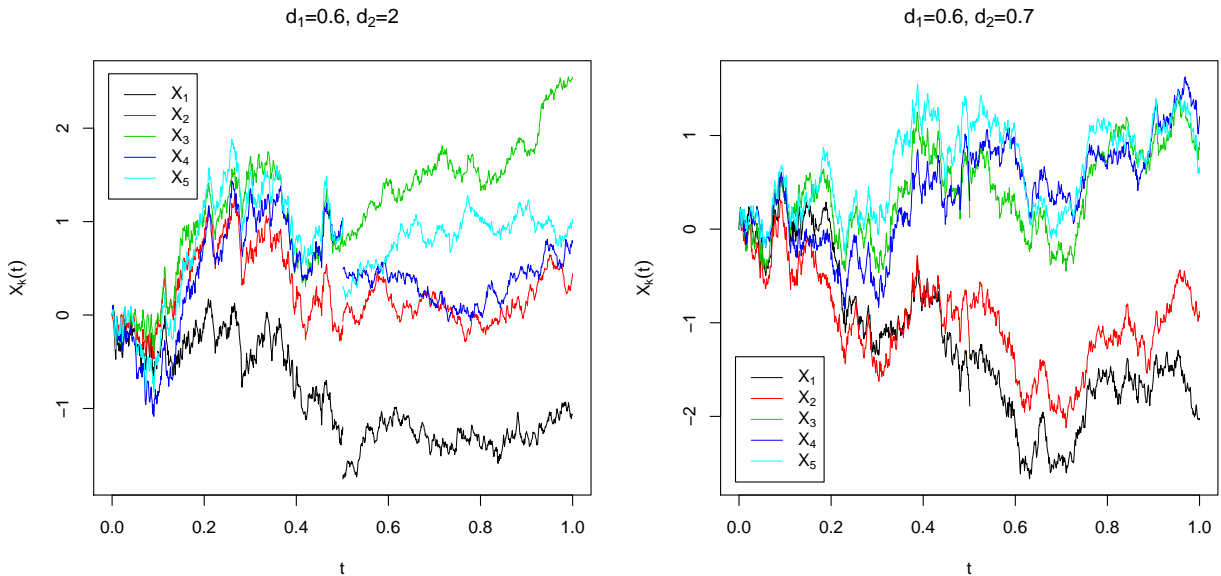}
\caption{Simulated sample paths of the random processes of the sequence $\cb{X_k}$ (horizontal axis is the index set $[0,1]$ of the random processes $\cb{X_k}$, vertical axis is the value of the random process $X_k(t)$ at a point $t\in[0,1]$)}
\end{figure}

The rest of the paper is organized as follows. The sequence of random processes $\cb{X_k}$ is defined and investigated in Section \ref{sec:sp}. In Section \ref{sec:clt} we establish sufficient conditions for the central limit theorem for the partial sums $\sum_{k=1}^nX_k$.

\section{Some preliminaries}\label{sec:sp}
Let $(S,\mathcal{S},\mu)$ be a $\sigma$-finite measure space. Consider a sequence of independent and identically distributed measurable random processes $\cb{\varepsilon_k}\mathrel{\mathop:}=\cb{\varepsilon_k(t): t\in S}$ defined on the same probability space $\pa{\Omega,\mathcal{F}, P}$ with $\on{E}\varepsilon_k(t)=0$ and $\on{E}\varepsilon_k^2(t)<\infty$ for each $t\in S$ and each $k\in\mathbb{Z}$. Let us denote $\sigma^2(t)\mathrel{\mathop:}=\on{E}\varepsilon_0^2(t)$ and $\sigma(s,t)\mathrel{\mathop:}=\on{E}\br{\varepsilon_0(s)\varepsilon_0(t)}$, where $s,t\in S$.

We define a sequence of random processes $\cb{X_k}\mathrel{\mathop:}=\cb{X_k(t),t\in S}$ by setting for each $t\in S$ and each $k\in\mathbb{Z}$
\begin{equation}
\label{eq:series of rv}
X_k(t)\mathrel{\mathop:}=\sum_{j=0}^\infty\pa{j+1}^{-d(t)}\varepsilon_{k-j}(t),
\end{equation}
where $d(t)>1/2$ for each $t\in S$. $d(t)>1/2$ is a necessary and sufficient condition for the almost sure convergence of the series \eqref{eq:series of rv}. This fact easily follows from Kolmogorov's three-series theorem.

It is possible to choose some other second-order stationary sequence of random variables that can have long memory (for example, $Y_k(t)=\sum_{j=0}^\infty a_j(t)\varepsilon_{k-j}(t)$, where $a_j(t)\sim (j+1)^{-d(t)}$), but our aim is to investigate space varying memory and we want to avoid any unnecessary technical difficulties.

If $\on{E}\varepsilon_0(t)\ne0$, then the sequence $\cb{X_k(t)}$, $t\in S$, can only have short memory (i.e.\ absolutely summable autocovariances), since then the series~\eqref{eq:series of rv} converges almost surely if and only if  $d(t)>1$ and absolute summability of $\pa{j+1}^{-d(t)}$ implies that autocovariances are absolutely summable (see, for example, \citet{hamilton1994}, p. 70).

Routine calculations show that the sequences $\cb{X_k(s)}$ and $\cb{X_k(t)}$ for $s,t\in S$ are sequences of zero mean random variables with the following expression for the cross-covariance
\begin{equation}
\label{eq:cross-cov}
\on{E}\br{X_0(s)X_h(t)}=\sigma(s,t)\sum_{j=0}^\infty\pa{j+1}^{-d(s)}\pa{j+h+1}^{-d(t)}.
\end{equation}

Now we establish the asymptotic behaviour of the sequence of cross-covariances. $a_n\sim b_n$ as $n\to\infty$ indicates that the sequences $a_n$ and $b_n$ are asymptotically equivalent, i.e.\ the ratio of the two sequences tends to one as $n$ goes to infinity.
\begin{prp}
\label{prp:asym cov}
Let $s,t\in S$ be fixed. If $1/2<d(s)<1$ and $d(t)>1/2$, then
\begin{equation*}
\on{E}\br{X_0(s)X_h(t)}\sim c(s,t)\sigma(s,t)\cdot h^{1-\br{d(s)+d(t)}}\quad\text{as}\quad h\to\infty,
\end{equation*}
where
\begin{equation}
\label{eq:function c}
c(s,t)\mathrel{\mathop:}=\int_0^\infty\!x^{-d(s)}(x+1)^{-d(t)}\,\mathrm{d}x.
\end{equation}
If $s=t$, we denote $c(t)\mathrel{\mathop:}=c(t,t)$ and $\sigma^2(t)\mathrel{\mathop:}=\sigma(t,t)$.

If $d(s)=d(t)=1$, then
\begin{equation*}
\on{E}\br{X_0(s)X_h(t)}\sim\sigma(s,t)\cdot h^{-1}\ln h\quad\text{as}\quad h\to\infty.
\end{equation*}
\end{prp}
\begin{proof}
We approximate the series in equation \eqref{eq:cross-cov} by integrals to obtain the following inequalities: if $\frac{1}{2}<d(s)<1$ and $d(t)>\frac{1}{2}$, then we obtain
\begin{align}
\label{in:series}
\sum_{j=0}^\infty\pa{j+1}^{-d(s)}\pa{j+h+1}^{-d(t)}&\ge h^{1-\br{d(s)+d(t)}}\int_{\frac1h}^\infty\!x^{-d(s)}\pa{x+1}^{-d(t)}\,\mathrm{d}x,\\
\notag
\sum_{j=0}^\infty\pa{j+1}^{-d(s)}\pa{j+h+1}^{-d(t)}&\le h^{1-\br{d(s)+d(t)}}\int_0^\infty\!x^{-d(s)}\pa{x+1}^{-d(t)}\,\mathrm{d}x;
\end{align}
if $d(s)=d(t)=1$, then we have that
\begin{align*}
\sum_{j=0}^\infty\br{(j+1)(j+h+1)}^{-1}&\ge h^{-1}\br{\ln\pa{\frac{h+1}2}+\int_1^\infty\![y(y+1)]^{-1}\,\mathrm{d}y}\\
\sum_{j=0}^\infty\br{(j+1)(j+h+1)}^{-1}&\le\pa{h+1}^{-1}+h^{-1}\br{\ln\pa{\frac{h+1}2}+\int_1^\infty\!\br{y(y+1)}^{-1}\,\mathrm{d}y}.\qedhere
\end{align*}
\end{proof}

Next, we investigate the convergence of the series of cross-covariances.
\begin{prp}
\label{prp:series of ac}
Let $s,t\in S$. The series
\begin{equation}
\label{eq:series of ac}
\sum_{h=1}^\infty\on{E}\br{X_0(s)X_h(t)}
\end{equation}
converges if and only if both of the conditions $d(t)>1$ and $d(s)+d(t)>2$ are fulfilled.
\end{prp}
\begin{proof}
Series \eqref{eq:series of ac} has the following expression
\begin{equation*}
\sum_{h=1}^\infty\on{E}\br{X_0(s)X_h(t)}=\sigma(s,t)\br{\sum_{h=1}^\infty\pa{h+1}^{-d(t)}+\sum_{h=1}^\infty\sum_{j=1}^\infty\pa{j+1}^{-d(s)}\pa{j+h+1}^{-d(t)}}.
\end{equation*}
The first series of the right-hand side of the equation above converges if and only if $d(t)>1$. Thus, we only need to investigate the convergence of the series
\begin{equation}
\label{eq:series}
\sum_{h=1}^\infty\sum_{j=1}^\infty\pa{j+1}^{-d(s)}\pa{j+h+1}^{-d(t)}.
\end{equation}

A slight modification of inequality \eqref{in:series} shows that series \eqref{eq:series} diverges if $d(s)+d(t)\le 2$. To show that series \eqref{eq:series} converges if $d(t)>1$ and $d(s)+d(t)>2$, we choose $\delta>0$ such that $1<1+\delta<d(t)$ and $d(s)+d(t)-\delta>2$ to obtain the following inequality
\begin{equation*}
\sum_{h=1}^\infty\sum_{j=1}^\infty\pa{j+1}^{-d(s)}\pa{j+h+1}^{-d(t)}\le\sum_{h=1}^\infty\sum_{j=1}^\infty\pa{j+1}^{-d(s)-d(t)+1+\delta}h^{-(1+\delta)}<\infty.\qedhere
\end{equation*}
\end{proof}
\begin{remark}
The series $\sum_{h=1}^\infty\on{E}\br{X_0(t)X_h(t)}$ converges if and only if $d(t)>1$.
\end{remark}

Suppose $\mathcal{L}^2\pa{\mu}\mathrel{\mathop:}=\mathcal{L}^2\pa{S,\mathcal S,\mu}$ is a separable space of real valued square $\mu$-integrable functions with a seminorm
\begin{equation*}
\nm{f}_2\mathrel{\mathop:}=\pa{\int_S\!\abs{f(r)}^2\,\mu\pa{\mathrm{d}r}}^{1/2}.
\end{equation*}

\mbox{Proposition \ref{prp:sqi}} establishes assumptions under which the sample paths of the processes $\cb{X_k}$ almost surely belong to the space $\mathcal{L}^2\pa{\mu}$.
\begin{prp}
\label{prp:sqi}
The sample paths of the random processes $\cb{X_k}$ almost surely belong to the space $\mathcal{L}^2\pa{\mu}$ if and only if both of the integrals
\begin{equation}
\label{eq:sample paths}
\quad\int_S\!\sigma^2(r)\,\mu(\mathrm{d}r)\quad\text{and}\quad\int_S\!\frac{\sigma^2(r)}{2d(r)-1}\,\mu(\mathrm{d}r)
\end{equation}
are finite.
\end{prp}
\begin{proof}
We show that the expected value
\begin{equation*}
\on{E}\br{\int_S\,X^2_0\pa{r}\!\mu\pa{\mathrm{d}r}}
\end{equation*}
is finite if and only if integrals \eqref{eq:sample paths} are finite. First, using Fubini's theorem we obtain
\begin{equation*}
\on{E}\br{\int_S\,X^2_0\pa{r}\!\mu\pa{\mathrm{d}r}}=\int_S\,\on{E}X^2_0\pa{r}\!\mu\pa{\mathrm{d}r}.
\end{equation*}
Secondly, setting $h=0$ and $s=t$ in equation \eqref{eq:cross-cov} gives the expression for the variance
\begin{equation*}
\on{E}X_0^2(t)=\sigma^2(t)\sum_{j=0}^\infty\pa{j+1}^{-2d(t)},\quad t\in S.
\end{equation*}
Approximation of the series above by integrals leads to the following inequalities
\begin{align*}
2\int_S\!\on{E}X_0^2(r)\,\mu\pa{\mathrm{d}r}&\ge\int_S\!\sigma^2(r)\,\mu\pa{\mathrm{d}r}+\int_S\!\frac{\sigma^2(r)}{2d(r)-1}\,\mu\pa{\mathrm{d}r},\\
\int_S\!\on{E}X_0^2(r)\,\mu\pa{\mathrm{d}r}&\le\int_S\!\sigma^2(r)\,\mu\pa{\mathrm{d}r}+\int_S\!\frac{\sigma^2(r)}{2d(r)-1}\,\mu\pa{\mathrm{d}r}.\qedhere
\end{align*}
\end{proof}

\section{The central limit theorem}\label{sec:clt}
Before we formulate sufficient conditions for the central limit theorem, we clarify what is meant by weak convergence of a sequence of random processes with square $\mu$-integrable sample paths (see \citet{cremers1986} for details). Suppose $\cb{\xi_n}$ is a sequence of measurable random processes with sample paths in $\mathcal{L}^2\pa{\mu}$. Let $L^2\pa{\mu}\mathrel{\mathop:}=L^2\pa{S,\mathcal{S},\mu}$ be the corresponding Banach space of equivalence classes of $\mu$-almost everywhere equal measurable functions and let $\mathcal{B}\pa{L^2\pa{\mu}}$ be its Borel $\sigma$-algebra. The maps $\hat{\xi}_n:\Omega\to L^2\pa{\mu},\omega\to\hat{\xi}_n(\omega)\mathrel{\mathop:}=\xi_n\pa{\omega,\cdot}$ are $\mathcal{F}-\mathcal{B}\pa{L^2\pa{\mu}}$-measurable (see \citet{cremers1986}); hence the distributions $\hat P_n$ of $\hat{\xi}_n$ are well defined probability measures on $\pa{L^2\pa{\mu},\mathcal{B}\pa{L^2\pa{\mu}}}$. It is said that the sequence $\cb{\xi_n}$ converges weakly to $\xi$ and it is written $\xi_n\Rightarrow\xi$ if the corresponding image measures $\hat P_n$ converge weakly, i.e.\ $\int f\mathrm{d}\hat P_n\to\int f\mathrm{d}\hat P$ for all continuous and bounded  real valued functions defined on $L^2\pa{\mu}$.

The following result establishes sufficient conditions for the central limit theorem for the partial sums $\sum_{k=1}^nX_k$.
\begin{prp}
\label{prp:clt}
\begin{enumerate}[(i)]
\item Suppose $1/2<d(t)<1$, $\on{E}\varepsilon_0^2(t)<\infty$ for each $t\in S$ and both of the integrals
\begin{equation}
\label{eq:intclt}
\int_S\!\frac{\sigma^2(r)}{\br{1-d(r)}^2}\,\mu\pa{\mathrm{d}r}\quad\text{and}\quad\int_S\!\frac{\sigma^2(r)}{\br{1-d(r)}\br{2d(r)-1}}\,\mu\pa{\mathrm{d}r}
\end{equation}
are finite. Then
\begin{equation*}
n^{-H}\sum_{k=1}^nX_k\Rightarrow\mathcal{G},
\end{equation*}
where $n^{-H}$ is a sequence of multiplication operators with the expression $n^{-H}f(t)=n^{-\br{3/2-d(t)}}f(t)$ for each $t\in S$, $f\in L^2\pa{\mu}$, and $\mathcal{G}\mathrel{\mathop:}=\cb{\mathcal{G}(t),t\in S}$ is a zero mean Gaussian random process with the following autocovariance function
\begin{equation*}
\on{E}\br{\mathcal{G}(s)\mathcal{G}(t)}=\frac{\br{c(s,t)+c(t,s)}\sigma(s,t)}{\pa{2-\br{d(s)+d(t)}}\pa{3-\br{d(s)+d(t)}}},
\end{equation*}
where $c(s,t)$ is the function \eqref{eq:function c} and $\sigma(s,t)\mathrel{\mathop:}=\on{E}\br{\varepsilon_0(s)\varepsilon_0(t)}$, $s,t\in S$.
\item Suppose $d(t)=1$ and $\on{E}\varepsilon_0^2(t)<\infty$ for each $t\in S$ and
\begin{equation*}
\int_S\!\sigma^2(r)\,\mu\pa{\mathrm{d}r}<\infty.
\end{equation*}
Then
\begin{equation*}
\frac1{\sqrt n\ln n}\sum_{k=1}^nX_k\Rightarrow\mathcal{G}',
\end{equation*}
where $\mathcal{G'}\mathrel{\mathop:}=\cb{\mathcal{G'}(t),t\in S}$ is a zero mean Gaussian random process with the autocovariance function $\on{E}\br{\mathcal{G}'(s)\mathcal{G}'(t)}=\sigma(s,t)$, where $\sigma(s,t)\mathrel{\mathop:}=\on{E}\br{\varepsilon_0(s)\varepsilon_0(t)}, s,t\in S$.
\end{enumerate}
\end{prp}

\begin{remark}
If the essential infimum of $d\pa{t}$ is greater than 1,  that is, if
\begin{equation*}
\sup\cb{x\in\mathbb{R}:\mu\pa{\cb{t:d\pa{t}<x}}=0}>1,
\end{equation*}
then we can use Theorem 1 of \citet{rackauskas2010} to show that the central limit theorem holds for the partial sums of a similar second-order stationary sequence of  $L^2\pa{\mu}$-valued random elements. Suppose that $\cb{\psi_k}$ is a sequence of independent and identically distributed random elements of $L^2\pa{\mu}$. Let us define a sequence of $L^2\pa{\mu}$-valued random elements by setting for each $k\in\mathbb{Z}$
\begin{equation}
\label{eq:lsqre}
Y_k\mathrel{\mathop:}=\sum_{j=0}^\infty A_j\psi_{k-j},
\end{equation}
where $A_j$ is a sequence of multiplication operators with the following expression
\begin{equation*}
A_j\psi_{k-j}\pa{t}=\pa{j+1}^{-d(t)}\psi_{k-j}(t)
\end{equation*}
for each $t\in S$. The sequence \eqref{eq:lsqre} is essentially similar to the sequence $\cb{X_k}$ of random processes \eqref{eq:series of rv}. According to Theorem~1 of \citet{rackauskas2010}, $n^{-1/2}\sum_{k=1}^nY_k$ converges in distribution to a Gaussian random element of $L^2\pa{\mu}$ if $n^{-1/2}\sum_{k=1}^n\psi_{k}$ converges in distribution to a Gaussian random element of $L^2\pa{\mu}$ (see \citet{rackauskas2010} for more details).
\end{remark}

\begin{proof}[Proof of Proposition \ref{prp:clt}]
The proof is based on a part of Theorem 2 of \citet{cremers1986}. We provide the proof in several steps. To establish that the sequence $\cb{\xi_n}$ weakly converges to $\xi$, we need to show that the following is true:
\begin{enumerate}[(a)]
\item\label{partA} for each $t\in S$ $\on{E}\xi_n^2(t)\to\on{E}\xi^2(t)$ as $n\to\infty$;
\item\label{partB} the finite-dimensional distributions of $\xi_n$ converge weakly to those of the $\xi$ almost everywhere;
\item\label{partC} for each $t\in S$ and $n\in\mathbb{N}$ $\on{E}\xi_n^2(t)\le f(t)$, where $f$ is a non-negative $\mu$-integrable function.
\end{enumerate}

We begin by proving part \eqref{partA}. We show that for each $t\in S$ the sequences $\on{E}\br{n^{-\br{3/2-d(t)}}\sum_{k=1}^nX_k(t)}^2$ and $\on{E}\br{(\sqrt n\ln n)^{-1}\sum_{k=1}^nX_k(t)}^2$ converge to $\on{E}\mathcal{G}^2(t)$ and $\on{E}\mathcal{G}'^2(t)$ respectively.

The growth rate of the cross-covariance of the partial sums of the sequences $\cb{X_k(s)}$ and $\cb{X_k(t)}, s,t\in S$, is established in Proposition \ref{prp:cps}.
\begin{prp}
\label{prp:cps}
If $1/2<d(s)<1$ and $1/2<d(t)<1$, then
\begin{equation*}
\on{E}\br{\pa{\sum_{k=1}^nX_k(s)}\pa{\sum_{k=1}^nX_k(t)}}\sim\frac{\br{c(s,t)+c(t,s)}\sigma(s,t)}{\pa{2-\br{d(s)+d(t)}}\pa{3-\br{d(s)+d(t)}}}\cdot n^{3-\br{d(s)+d(t)}}\quad\text{as}\quad n\to\infty,
\end{equation*}
where $c(s,t)$ is function \eqref{eq:function c}.

If $d(s)=1$ and $d(t)=1$, then
\begin{equation*}
\on{E}\br{\pa{\sum_{k=1}^nX_k(s)}\pa{\sum_{k=1}^nX_k(t)}}\sim\sigma(s,t)\cdot n\ln^2n.
\end{equation*}
\end{prp}
\begin{proof}
The cross-covariance of the partial sums of the sequences $\cb{X_k(s)}$ and $\cb{X_k(t)}$ has the following expression
\begin{multline}
\label{eq:pscov}
\on{E}\br{\pa{\sum_{k=1}^nX_k(s)}\pa{\sum_{k=1}^nX_k(t)}}=n\on{E}\br{X_0(s)X_0(t)}\\+\sum_{k=1}^{n-1}\sum_{l=k+1}^n\on{E}\br{X_k(s)X_l(t)}+\sum_{k=1}^{n-1}\sum_{l=k+1}^n\on{E}\br{X_k(t)X_l(s)}.
\end{multline}
Since
\begin{equation*}
\sum_{k=1}^{n-1}\sum_{l=k+1}^n\on{E}\br{X_k(s)X_l(t)}=n\sum_{k=1}^{n-1}\on{E}\br{X_0(s)X_k(t)}-\sum_{k=1}^{n-1}k\on{E}\br{X_0(s)X_k(t)},
\end{equation*}
we can use the results of Proposition \ref{prp:asym cov} to obtain the following asymptotic relations: if $1/2<d(s)<1$ and $1/2<d(t)<1$, then
\begin{align*}
\sum_{k=1}^{n-1}\on{E}\br{X_0(s)X_k(t)}
    &\sim\frac{c(s,t)\sigma(s,t)}{2-\br{d(s)+d(t)}}\cdot n^{2-\br{d(s)+d(t)}},\\
\sum_{k=1}^{n-1}k\on{E}\br{X_0(s)X_k(t)}
    &\sim\frac{c(s,t)\sigma(s,t)}{3-\br{d(s)+d(t)}}\cdot n^{3-\br{d(s)+d(t)}};
\end{align*}
if $d(s)=1$ and $d(t)=1$, then
\begin{align*}
\sum_{k=1}^{n-1}\on{E}\br{X_0(s)X_k(t)}
    &\sim\frac{\sigma(s,t)}2\cdot \ln^2n,\\
\sum_{k=1}^{n-1}k\on{E}\br{X_0(s)X_k(t)}
    &\sim\sigma(s,t)\cdot n\ln n.\qedhere
\end{align*}
\end{proof}
\begin{remark}
\label{remark:var}
The growth rate of the variance of the partial sums of the sequence $\cb{X_k(t)}$ is the following: if $1/2<d(t)<1$, then
\begin{align*}
\on{E}\br{\sum_{k=1}^nX_k(t)}^2&\sim\frac{c(t)\sigma^2(t)}{\br{1-d(t)}\br{3-2d(t)}}\cdot n^{3-2d(t)};
\intertext{if $d(t)=1$, then}
\on{E}\br{\sum_{k=1}^nX_k(t)}^2&\sim\sigma^2(t)\cdot n\ln^2n.
\end{align*}
\end{remark}
Remark \ref{remark:var} completes the proof of part \eqref{partA}.

Now we move on to the proof of part \eqref{partB} and show that the finite dimensional distributions of $n^{-H}\sum_{k=1}^nX_k$ and $(\sqrt{n}\ln n)^{-1}\sum_{k=1}^nX_k$ converge to those of $\mathcal G$ and $\mathcal G'$ respectively. In order to prove the convergence of finite dimensional distributions we investigate a sequence of random vectors
\begin{equation}
\label{eq:rv}
\left(\begin{array}{ccc}b_n^{-1}(t_1)\sum_{k=1}^nX_k(t_1) & \ldots & b_n^{-1}(t_q)\sum_{k=1}^nX_k(t_q)\end{array}\right),
\end{equation}
where $t_1,\ldots,t_q\in S$ and
\begin{equation*}
b_n(t)\mathrel{\mathop:}=
\begin{cases}
n^{3/2-d(t)},&1/2<d(t)<1;\\
\sqrt{n}\ln n,&d(t)=1.
\end{cases}
\end{equation*}

The sum of dependent random variables $\sum_{k=1}^nX_k(t)$ can be expressed as a series of independent random variables: if $n\ge2$, then we have the following identity
\begin{equation*}
\sum_{k=1}^nX_k(t)=\sum_{j=-\infty}^nz_{n,j}(t)\varepsilon_j(t),
\end{equation*}
where
\begin{equation*}
z_{n,j}(t)\mathrel{\mathop:}=
\begin{cases}
\sum_{k=1}^{n-j+1}k^{-d(t)},&2\le j\le n;\\
\sum_{k=1}^n\pa{k-j+1}^{-d(t)},& j<2.
\end{cases}
\end{equation*}
By denoting
\begin{equation*}
\varepsilon_j^{(q)}\mathrel{\mathop:}=\left(\begin{array}{ccc}\varepsilon_j(t_1) & \ldots & \varepsilon_j(t_q)\end{array}\right)^{\mathrm{T}}
\end{equation*}
and
\begin{equation*}
B_{n,j}\mathrel{\mathop:}=\on{diag}\left(\begin{array}{ccc}b_n^{-1}(t_1)z_{n,j}(t_1) & \ldots & b_n^{-1}(t_q)z_{n,j}(t_q)\end{array}\right),
\end{equation*}
we can express the sequence of random vectors \eqref{eq:rv} compactly as
\begin{equation*}
\sum_{j=-\infty}^nB_{n,j}\varepsilon_j^{(q)}=\left(\begin{array}{c}\sum_{j=-\infty}^n\br{b_n^{-1}(t_1)z_{n,j}(t_1)}\varepsilon_j(t_1) \\\hdotsfor{1} \\\sum_{j=-\infty}^n\br{b_n^{-1}(t_q)z_{n,j}(t_q)}\varepsilon_j(t_q)\end{array}\right).
\end{equation*}
Using the fact that the operator norm of a diagonal matrix is the largest entry in absolute value, we obtain the operator norm of the matrix $B_{n,j}$
\begin{equation*}
\nm{B_{n,j}}_M=\max_{1\le i\le q}\abs{b_n^{-1}(t_i)z_{n,j}(t_i)}.
\end{equation*}

Now suppose that
\begin{equation*}
\gamma_j^{(q)}\mathrel{\mathop:}=\left(\begin{array}{ccc}\gamma_j(t_1) & \ldots & \gamma_j(t_q)\end{array}\right)^{\mathrm{T}}
\end{equation*}
is a zero mean normal random vector with the same covariance matrix as the vector $\varepsilon_j^{(q)}$. The sequence $\sum_{j=-\infty}^nB_{n,j}\gamma_j^{(q)}$ converges in distribution to a normal random vector if and only if the sequence of covariance matrices converges. Clearly, it follows from the results of Proposition \ref{prp:cps}, that the sequence of covariance matrices converges.

To prove that the finite dimensional distributions converge weakly, we apply Proposition 4.1 of \citet{rackauskas2011}, which establishes that under certain conditions the sequences $\sum_{j=-\infty}^nB_{n,j}\varepsilon_j^{(q)}$ and $\sum_{j=-\infty}^nB_{n,j}\gamma_j^{(q)}$ have the same limiting behaviour in the sense that if one converges in distribution then so does the other and their limits coincide.

We consider for $q$-dimensional random vectors $U$ and $V$ the distance
\begin{equation*}
\zeta_3\pa{U,V}\mathrel{\mathop:}=\sup_{f\in\mathcal{F}_3}\abs{\on{E}f\pa{U}-\on{E}f\pa{V}},
\end{equation*}
where $\mathcal{F}_3$ is the set of three times Frechet differentiable functions $f:\mathbb{R}^q\to\mathbb{R}$ such that
\begin{equation*}
\sup_{x\in \mathbb{R}^q}\abs{f^{(j)}(x)}\le1,\quad\text{for}\ j=0,\ldots,3.
\end{equation*}

For the sake of convenience, we state Proposition 4.1 of \citet{rackauskas2011} here.
\begin{prp}
\label{prp:rs}
If the following two conditions
\begin{equation}
\label{cond1}
\lim_{n\to\infty}\sup_{j\in\mathbb{Z}}\nm{B_{n,j}}_M=0
\end{equation}
and
\begin{equation}
\label{cond2}
\limsup_{n\to\infty}\sum_{j\in\mathbb{Z}}\nm{B_{n,j}}_M^2<\infty.
\end{equation}
are satisfied, then
\begin{equation*}
\lim_{n\to\infty}\zeta_3\pa{\sum_{j=-\infty}^nB_{n,j}\varepsilon_j^{(q)},\sum_{j=-\infty}^nB_{n,j}\gamma_j^{(q)}}=0.
\end{equation*}
\end{prp}

Propositon \ref{prp:on} establishes that the sequence of operator norms of the matrices $B_{n,j}$ satisfies two properties that are needed to apply Proposition \ref{prp:rs}.
\begin{prp}
\label{prp:on}
If $\frac{1}{2}<d(t)\le1$ for each $t\in S$, then both of conditions \eqref{cond1} and \eqref{cond2} are satisfied.
\end{prp}
\begin{proof}
To prove that condition \eqref{cond1} holds, we first notice that
\begin{equation*}
\sup_{j\in\mathbb{Z}}\nm{B_{n,j}}_M=\max_{1\le i\le q}\abs{b_n^{-1}(t_i)z_{n,1}(t_i)}
\end{equation*}
and then we use the following asymptotic relations: if $\frac{1}{2}<d(t)<1$, then we have that
\begin{equation*}
\sum_{k=1}^nk^{-d(t)}\sim\frac{n^{1-d(t)}}{1-d(t)};
\end{equation*}
if $d(t)=1$, then we obtain
\begin{equation*}
\sum_{k=1}^nk^{-1}\sim\ln n.
\end{equation*}

We use the following expression to prove that condition \eqref{cond2} holds
\begin{equation*}
\sum_{j=-\infty}^nz_{n,j}^2(t)=\sum_{j=2}^n\br{\sum_{k=1}^{n-j+1}k^{-d(t)}}^2+\sum_{j=0}^\infty\br{\sum_{k=1}^n\pa{k+j}^{-d(t)}}^2.
\end{equation*}
Routine approximations of sums by integrals from above lead to the following inequalities: if $\frac{1}{2}<d(t)<1$, then we have that
\begin{equation*}
\sum_{j=2}^n\br{\sum_{k=1}^{n-j+1}k^{-d(t)}}^2\le\frac{1}{\br{1-d(t)}^2[3-2d(t)]}\br{n^{3-2d(t)}-1};
\end{equation*}
if $d(t)=1$, then we obtain
\begin{equation*}
\sum_{j=2}^n\br{\sum_{k=1}^{n-j+1}k^{-1}}^2\le(n-1)+n\ln^2n.
\end{equation*}

To prove that
\begin{equation*}
\lim_{n\to\infty}\frac{1}{n^{3-2d(t)}}\sum_{j=0}^\infty\br{\sum_{k=1}^n\pa{k+j}^{-d(t)}}^2<\infty
\end{equation*}
for $\frac{1}{2}<d(t)\le1$, we first divide the series in the expression above into two summands
\begin{equation}
\label{eq:series div}
\sum_{j=0}^\infty\br{\sum_{k=1}^n\pa{k+j}^{-d(t)}}^2=\br{\sum_{k=1}^nk^{-d(t)}}^2+\sum_{j=1}^\infty\br{\sum_{k=1}^n\pa{k+j}^{-d(t)}}^2
\end{equation}
and approximate the first summand on the right-hand side of the equation \eqref{eq:series div} by integral from above
\begin{equation*}
\sum_{k=1}^{n}k^{-d(t)}\le
\begin{cases}
\frac{n^{1-d(t)}}{1-d(t)}&\text{if}\quad\frac{1}{2}<d(t)<1;\\
1+\ln n &\text{if}\quad d(t)=1.
\end{cases}
\end{equation*}
We express the second summand on the right-hand side of equation \eqref{eq:series div} in the following way
\begin{equation*}
\frac{1}{n^{3-2d(t)}}\sum_{j=1}^\infty\br{\sum_{k=1}^{n}(k+j)^{-d(t)}}^2=\sum_{i=1}^\infty\frac{1}{n}\sum_{j=(i-1)n+1}^{in}\br{\frac{1}{n}\sum_{k=1}^{n}\pa{\frac{k}{n}+\frac{j}{n}}^{-d(t)}}^2
\end{equation*}
and the interchange of limits leads to the result
\begin{align*}
&\lim_{n\to\infty}\frac{1}{n^{3-2d(t)}}\sum_{j=1}^\infty\br{\sum_{k=1}^n(k+j)^{-d(t)}}^2=\int_0^\infty\!\br{\int_0^1\!(s+u)^{-d(t)}\,\mathrm{d}s}^2\,\mathrm{d}u<\infty.\qedhere
\end{align*}
\end{proof}
The proof of part \eqref{partB} is complete.

Finally, we prove part \eqref{partC}. Proposition \ref{prp:dom} establishes the existence of non-negative $\mu$-integrable functions that dominate the sequence of the variance of the partial sums.
\begin{prp}
\label{prp:dom}
Let $t\in S$. If $\frac12<d(t)<1$, then
\begin{equation}
\label{eq:domfun}
\on{E}\br{\frac{1}{n^{3/2-d(t)}}\sum_{k=1}^nX_k(t)}^2\le\sigma^2(t)\br{1+\frac{1}{2d(t)-1}}+\frac{\sigma^2(t)c(t)}{\br{1-d(t)}\br{3-2d(t)}},
\end{equation}
where $c(t)$ is the function \eqref{eq:function c}. If $d(t)=1$ for each $t\in S$, then
\begin{equation}
\label{eq:domfun1}
\on{E}\br{\frac{1}{\sqrt{n}\ln n}\sum_{k=1}^nX_k(t)}^2\le C\cdot \sigma^2(t),
\end{equation}
where $C$ is a positive constant.
\end{prp}
\begin{proof}
To establish the first inequality in Proposition \ref{prp:dom}, we set $s=t$ in expression \eqref{eq:pscov} and approximate the sums in expression \eqref{eq:pscov} by integrals from above.

The following reasoning leads to inequality \eqref{eq:domfun1}. We set $s=t$ in expression $\eqref{eq:pscov}$ to obtain an expression for the left-hand side of inequality \eqref{eq:domfun1}. Since $d(t)=1$ for each $t\in S$, by setting $s=t$ in expression  $\eqref{eq:cross-cov}$, we see that the only term in the expression of the left-hand side of inequality \eqref{eq:domfun1} that depends on $t$ is $\sigma^2(t)$. It follows that the sequence
\begin{equation*}
\frac1{\sigma^2(t)}\cdot\on{E}\br{\frac{1}{\sqrt{n}\ln n}\sum_{k=1}^nX_k(t)}^2
\end{equation*}
is a convergent sequence (see Remark \ref{remark:var}) which does not depend on $t$. So it is bounded by some positive constant, say $C$.\qedhere
\end{proof}
\begin{remark}
\label{c(t)}
We can easily obtain the following upper bound for the function $c(t)$,
\begin{align*}
c(t)\le\frac1{1-d(t)}+\frac1{2d(t)-1},\quad t\in S,
\end{align*}
and then we have the following inequality
\begin{equation*}
\frac{\sigma^2(t)c(t)}{\br{1-d(t)}\br{3-2d(t)}}\le\frac{\sigma^2(t)}{\br{1-d(t)}^2}+\frac{\sigma^2(t)}{\br{1-d(t)}\br{2d(t)-1}}.
\end{equation*}
If integrals \eqref{eq:intclt} are finite, then the right-hand side of the inequality \eqref{eq:domfun} is a $\mu$-integrable function.
\end{remark}
Remark \ref{c(t)} completes the proof of part \eqref{partC}. The proof of Proposition \ref{prp:clt} is complete.\qedhere
\end{proof}

\section*{Acknowledgements}
This research was supported in part by the Research Council of Lithuania, grant No. MIP-053/2012.

\bibliography{bibliography}

\end{document}